\documentclass{birkjour}
 \newtheorem{thm}{Theorem}[section]
 \newtheorem{cor}[thm]{Corollary}
 
 \newtheorem{prop}[thm]{Proposition}
 \theoremstyle{definition}
 \newtheorem{defn}[thm]{Definition}
 \newtheorem{ex}[thm]{Example}
 \theoremstyle{remark}
 \newtheorem{rem}[thm]{Remark}
 \numberwithin{equation}{section}
\usepackage{color}
\begin{document}
%
%
%
%
%
%
%
%
%

\title[Weaving Hilbert space frames and duality]
 {Weaving Hilbert Space Frames and Duality}

\author[F. Arabyani Neyshaburi]{Fahimeh Arabyani Neyshaburi}
\address{Department of Mathematics and Computer Sciences, Hakim Sabzevari University, Sabzevar, Iran.}
\email{fahimeh.arabyani@gmail.com}

\author[A. Arefijamaal]{Ali Akbar Arefijamaal}
\address{Department of Mathematics and Computer Sciences, Hakim Sabzevari University, Sabzevar, Iran.}
\email{arefijamaal@hsu.ac.ir}


\begin{abstract}
Weaving Hilbert space frames have been introduced recently by Bemrose et al. to deal with some problems in  distributed  signal processing.
In this paper, we    survey this topic from the viewpoint of the duality principle, so we obtain  new properties  in weaving frame theory related to dual frames. Specifically, we give some sufficient  conditions  under which  a frame with its canonical dual, alternate duals or approximate duals constitute some concrete  pairs of  woven frames. Moreover, we survey the  stability of weaving frames under  perturbations and different operators.
\end{abstract}

\maketitle

\smallskip

\section{Introduction}
Today, frame theory   have been developed   in various areas of pure and applied science  and engineering \cite{Ar13,Ben06, Bod05, Bol98, Can04}.  Due to redundance  requirement in applications and theoretical goals, in over the past two decades, several generalizations of
frames have been presented \cite{Ali93, arefi3, arefi1,  Cas04, gav07,  Kaftal, Sun06}. Recently,   Bemrose et al.  \cite{Bemros}  have introduced a new concept in frame  theory  due to some   problems in distributed signal processing,  called \textit{weaving Hilbert space frames}. This notion  also  has potential applications in wireless  sensor networks which require distributed processing using different frames.
 Two given frames $\{\varphi_{i}\}_{i\in I}$ and $\{\psi_{i}\}_{i\in I}$ for a separable  Hilbert space  $\mathcal{H}$ are woven if there exist constants $0 <A \leq B< \infty$ so that for every subset $J\subset I$ the family $\{\varphi_{i}\}_{i\in J}\cup \{\psi_{i}\}_{i\in J^{c}}$  is a frame for $\mathcal{H}$ with frame bounds $A$ and $B$.  In \cite{Bemros, lynch} the authors presented some remarkable properties of weaving frames  and weaving Riesz bases. In this paper, we  give new basic properties of  weaving frames related to dual frames to  survey under which conditions a frame with its dual  constitute   woven frames.
Moreover, we  present some approaches for constructing  concrete pairs of  woven frames.

The paper is organized as follows. In Section 2, we will review some basic definitions and results related to frames and weaving frames in Hilbert spaces. Section 3  is devoted to presenting several methods of  constructing  concrete pairs of  woven frames. We also obtain some   conditions to prove that  a frame with its duals or approximate duals constitute some pairs of  woven frames. As a consequence, we show that every Riesz basis is woven with its  canonical dual.
Finally, we survey the stability of woven frames under small perturbations in Section 4.

\section{Review and preliminaries}
A family of vectors $\varphi:=\{\varphi_{i}\}_{i\in I}$ in a separable Hilbert space   $\mathcal{H}$ is called a Riesz basis if $\overline{span}\{\varphi_{i}\}_{i\in I} = \mathcal{H}$ and   there exist the constants
 $0<A_{\varphi}\leq B_{\varphi}<\infty$ so that for all $\{c_{i}\}_{i\in I}\in l^{2}(I)$,
 \begin{eqnarray*}
A_{\varphi}\sum_{i\in I}\vert c_{i}\vert^{2} \leq \Vert \sum_{i\in I} c_{i}\varphi_{i}\Vert^{2} \leq B_{\varphi} \sum_{i\in I}\vert c_{i}\vert^{2}.
\end{eqnarray*}
The constants $A_{\varphi}$ and $B_{\varphi}$ are called \textit{Riesz basis bounds}. Even though, Riesz bases are useful tools in some applications, they present  a unique decomposition of signals. Redundancy is one of the striking properties of frame, which helps us    obtain infinitely many  decompositions of a signal and
 reduce the errors of reconstruction  formulas.
\begin{defn}A family of vectors
$\varphi:=\{\varphi_{i}\}_{i\in I}$ in   $\mathcal{H}$ is called a  \textit{frame} for
$\mathcal{H}$ if there exist the constants
 $0<A_{\varphi}\leq B_{\varphi}<\infty$ such that
\begin{eqnarray}\label{Def frame}
A_{\varphi}\|f\|^{2}\leq \sum_{i\in I}\vert f,\varphi_{i}\rangle\vert^{2}\leq
B_{\varphi}\|f\|^{2},\qquad (f\in \mathcal{H}).
\end{eqnarray}
\end{defn}
The constants $A_{\varphi}$ and $B_{\varphi}$ are called \textit{frame bounds}. The supremum of all lower frame bounds is called the \textit{optimal lower frame bound} and likewise, the \textit{optimal upper frame bound} is defined as the infimum of all upper frame bounds.  $\{\varphi_{i}\}_{i\in I}$ is said to be  a \textit{Bessel sequence} whenever in  (\ref{Def frame}), the right-hand side holds.
 A frame
 $\{\varphi_{i}\}_{i\in I}$ is called   $A$-\textit{tight frame} if $A=A_{\varphi}=B_{\varphi}$, and in the case of $A_{\varphi}=B_{\varphi}=1$ it is called a  \textit{Parseval frame}. Moreover,  $\varphi$ is an  \textit{$\epsilon$-nearly Parseval frame} \cite{BC}  if
\begin{eqnarray*}
(1-\epsilon)\|f\|^{2}\leq \sum_{i\in I}|\langle f,\varphi_{i}\rangle|^{2}\leq
(1+\epsilon)\|f\|^{2},\qquad (f\in \mathcal{H}).
\end{eqnarray*}
 Also, $\varphi$  is called  \textit{equal-norm} if  there exists $c>0$ so that  $\Vert \varphi_{i}\Vert=c$ for all $i\in I$ and it is   \textit{$\epsilon$-nearly equal norm} if
\begin{eqnarray*}
(1-\epsilon)c\leq \Vert \varphi_{i}\Vert\leq
(1+\epsilon)c.
\end{eqnarray*}

Given a
frame $\varphi=\{\varphi_{i}\}_{i\in I}$, the \textit{frame operator} is defined by
\begin{eqnarray*}
S_{\varphi}f=\sum_{i\in I}\langle f,\varphi_{i}\rangle \varphi_i.
\end{eqnarray*}
 It is a bounded, invertible, and self-adjoint
operator \cite{Chr08}.
Also, the \textit{synthesis operator} $T_{\varphi}: l^{2}\rightarrow \mathcal{H}$ is defined  by $T_{\varphi}\lbrace c_{i}\rbrace= \sum_{i\in I} c_{i}\varphi_{i}$. The frame operator can be writen as $S_{\varphi}= T_{\varphi}T_{\varphi}^{*}$ where $T_{\varphi}^{*}: \mathcal{H}\rightarrow l^{2}$, the adjoint of $T$, given by $T_{\varphi}^{*}\varphi= \lbrace \langle \varphi,\varphi_{i}\rangle\rbrace_{i\in I}$  is called the \textit{analysis operator}.
The family $\{S_{\varphi}^{-1}\varphi_{i}\}_{i\in I}$ is also a frame for $\mathcal{H}$,   called the \textit{canonical dual} frame. In general, a frame $\{\psi_i\}_{i\in I}\subseteq\mathcal{H}$ is called an \textit{alternate dual} or simply a \textit{dual}
 for   $\{\varphi_{i}\}_{i\in I}$ if
\begin{eqnarray}\label{Def Dual}
f=\sum_{i\in I}\langle f,\psi_{i}\rangle \varphi_i ,\qquad (f\in
\mathcal{H}).
\end{eqnarray}
All frames have at least a dual, the canonical dual, and redundant frames have infinite many  alternate  dual frames.
It is  known that  every dual frame is of the form  $\{S_{\varphi}^{-1}\varphi_{i}+u_{i}\}_{i\in I}$, where   $\{u_{i}\}_{i\in I}$ is a Bessel sequence that  satisfies
\begin{eqnarray}\label{dual1}
\sum_{i\in I}\langle f,u_{i}\rangle \varphi_{i}=0 ,\qquad (f\in
\mathcal{H}).
\end{eqnarray}

\smallskip

Now, let us to  present  some preliminaries about  weaving frames which are used in our main results.

\begin{defn}
A finite family of frames $\{\phi_{ij}\}_{j=1,i\in I}^{M}$ in Hilbert space $\mathcal{H}$ is said  to be \textit{woven} if there are universal constants $A$ and $B$ so that for every partition $\{J_{j}\}_{j=1}^{M}$ of $I$, the family   $\{\phi_{ij}\}_{j=1,i\in J_{j}}^{M}$ is a frame for $\mathcal{H}$ with bounds $A$ and $B$, respectively. Each family $\{\phi_{ij}\}_{j=1,i\in J_{j}}^{M}$ is called a \textit{weaving}.
\end{defn}
Moreover the family $\{\phi_{ij}\}_{j=1,i\in I}^{M}$ is called \textit{weakly woven} if for every partition  $\{J_{j}\}_{j=1}^{M}$ of $I$, the family $\{\phi_{ij}\}_{j=1,i\in J_{j}}^{M}$ is a frame for $\mathcal{H}$.
The following theorem shows that weakly woven is equivalent to the frames being woven.
\begin{thm}\label{2.0}\cite{Bemros}
Given two frames  $\{\varphi_{i}\}_{i\in I}$ and $\{\psi_{i}\}_{i\in I}$ for $\mathcal{H}$, the following are equivalent:
\begin{itemize}
\item[(i)] The two frames are woven.
\item[(ii)] The two frames are weakly woven.
\end{itemize}
 \end{thm}
  In the following we mention some properties, proved in \cite{Bemros}, which are used in the present paper.
\begin{prop}\label{2.4..}
If $\{\phi_{ij}\}_{j=1,i\in I}^{M}$ is a family of Bessel sequences for $\mathcal{H}$ with a Bessel bound $B_{j}$ for all $1\leq j \leq M$. Then every weaving is a Bessel sequence with a Bessel bound $\sum_{j=1}^{M}B_{j}$.
 \end{prop}
\begin{prop}\label{prop1.4}
Let $\varphi=\{\varphi_{i}\}_{i\in I}$ be a frame and $T$   an invertible operator satisfying
$\Vert I_{\mathcal{H}} - T\Vert^{2}  < \dfrac{A_{\varphi}}{B_{\varphi}}$.
Then $\varphi$ and $T\varphi$ are woven   with the  universal  lower bound $\left(\sqrt{A_{\varphi}}-\sqrt{B_{\varphi}}\Vert I_{\mathcal{H}}-T^{*}\Vert\right)^{2}$.
 \end{prop}
\begin{defn}
If $W_{1}$ and $W_{2}$ are subspaces of $\mathcal{H}$, let
\begin{eqnarray*}
d_{W_{1}}(W_{2})=\inf\{\Vert f-g\Vert : \quad f\in W_{1}, g\in \mathcal{S}_{W_{2}}\},
\end{eqnarray*}
and
\begin{eqnarray*}
d_{W_{2}}(W_{1})=\inf\{\Vert f-g\Vert : \quad f\in \mathcal{S}_{W_{1}}, g\in  W_{2} \},
\end{eqnarray*}
where $\mathcal{S}_{W_{i}}=\mathcal{S}_{\mathcal{H}}\cap W_{i}$ and $\mathcal{S}_{\mathcal{H}}$ is the unit sphere in $\mathcal{H}$. Then the distance between $W_{1}$ and $W_{2}$  is defined as
\begin{eqnarray*}
d(W_{1}, W_{2})= \min\{d_{W_{1}}(W_{2}), d_{W_{2}}(W_{1})\}
\end{eqnarray*}
\end{defn}
\begin{thm}\label{distance}
If $\varphi=\{\varphi_{i}\}_{i\in I}$ and $\psi=\{\psi_{i}\}_{i\in I}$ are two Riesz bases for $\mathcal{H}$, then the following are equivalent
\begin{itemize}
\item[(i)]
$\varphi$ and $\psi$ are woven.
\item[(ii)]
For every  $J\subset I$, $d(\overline{span}\{\varphi_{i}\}_{i\in J}, \overline{span}\{\psi_{i}\}_{i\in J^{c}})>0$.
\item[(iii)]There is a constant $c>0$ so that for every $J\subset I$,
\begin{eqnarray*}
d_{\varphi_{J},\psi_{J^{c}}} :=d(\overline{span}\{\varphi_{i}\}_{i\in J}, \overline{span}\{\psi_{i}\}_{i\in J^{c}})\geq c.
\end{eqnarray*}
\end{itemize}
\end{thm}
Throughout this paper, we suppose  $\mathcal{H}$ is a separable Hilbert space, $\mathcal{H}_{m}$ an $m$-dimensional Hilbert space, $I$ a countable index set and $I_{\mathcal{H}}$ is the identity operator on $\mathcal{H}$. For  every $J\subset I$, we show the complement of $J$ by $J^{c}$. Also,  we use of $T_{\varphi_{J}}$ and  $S_{\varphi_{J}}$ to denote the synthesis operator and frame operator of a frame $\varphi$, whenever the index set is  limited to $J\subset I$ and     we use of $[n]$ to denote the set $\{1, 2, ..., n\}$..

\section{Weaving and duality}
In this section, we  concentrate on examining some conditions  under which  a frame with its duals (approximate duals) are woven. This approach  helps us  construct some concrete pairs of woven frames.
\begin{thm}\label{alter dual}
Suppose that $\varphi:=\{\varphi_{i}\}_{i\in I}$ is a redundant  frame so that
\begin{eqnarray}\label{imp4}
\Vert I_{\mathcal{H}}-S_{\varphi}^{-1} \Vert^{2} < \dfrac{A_{\varphi}}{B_{\varphi}}.
\end{eqnarray}
 Then $\varphi$ has  infinitely many  dual frames    $\{\psi_{i}\}_{i\in I}$  for which $\{\varphi_{i}\}_{i\in I}$ and $\{\psi_{i}\}_{i\in I}$ are woven.
\end{thm}
\begin{proof}
Using the assumption with the aid of  Proposition \ref{prop1.4} implies that  $\varphi$ and $\{S_{\varphi}^{-1}\varphi_{i}\}_{i\in I}$ are woven frames for $\mathcal{H}$  with the universal  lower bound $\mathcal{A}:=\left(\sqrt{A_{\varphi}}-\sqrt{B_{\varphi}}\Vert I_{\mathcal{H}}-S_{\varphi}^{-1} \Vert\right)^{2}$. Now, let  $U=\{u_{i}\}_{i\in I}$ be a Bessel sequence, which satisfies (\ref{dual1}) and take $\epsilon>0$ so that
\begin{eqnarray}\label{inequality U}
\epsilon^{2} B_{U}+2\epsilon\sqrt{ B_{U}/A_{\varphi}} < \mathcal{A}.
\end{eqnarray}
Then  $\Psi_{\alpha}:=\{S_{\varphi}^{-1}\varphi_{i}+\alpha u_{i}\}_{i\in I}$ is a dual frame of $\varphi$,  for all $0<\alpha<\epsilon$. To show $\varphi$ and $\Psi_{\alpha}$ constitute  woven frames for $\mathcal{H}$, due to Proposition \ref{2.4..} we only need to prove the existence of a universal lower bound.
Suppose $J\subset I$  then
\begin{eqnarray*}
&&\sum_{i\in J}|\langle f,\varphi_{i}\rangle|^{2}+\sum_{i\in J^{c}}|\langle f,S_{\varphi}^{-1}\varphi_{i}+\alpha u_{i}\rangle|^{2}\\
&=& \sum_{i\in J}|\langle f,\varphi_{i}\rangle|^{2}+\sum_{i\in J^{c}}|\langle f,S_{\varphi}^{-1}\varphi_{i}\rangle+ \langle f,\alpha u_{i}\rangle|^{2}\\
&\geq&   \sum_{i\in J}|\langle f,\varphi_{i}\rangle|^{2}+\sum_{i\in J^{c}}\left||\langle f,S_{\varphi}^{-1}\varphi_{i}\rangle|-| \langle f, \alpha u_{i}\rangle|\right|^{2}\\
&\geq&  \sum_{i\in J}|\langle f,\varphi_{i}\rangle|^{2}+\sum_{i\in J^{c}}|\langle f,S_{\varphi}^{-1}\varphi_{i}\rangle|^{2}\\
&-&\sum_{i\in J^{c}}|\langle f, \alpha u_{i}\rangle|^{2}-2 \sum_{i\in J^{c}}|\langle f,S_{\varphi}^{-1}\varphi_{i}\rangle|| \langle f,\alpha u_{i}\rangle|\\
&\geq& \left(\mathcal{A}-\alpha^{2}B_{U}-2\alpha\sqrt{B_{U}/A_{\varphi}}\right)\Vert f\Vert^{2},
\end{eqnarray*}
 and so the lower bound is obtained by (\ref{inequality U}).
\end{proof}
Recall that the notion of approximate  dual for discrete frames introduced by Christensen and Laugesen in \cite{app.} due to the explicit computation of dual frames seems rather intricate. Approximate duals  are not only easier to construct
but also  lead to perfect reconstructions.
 Let $\varphi = \lbrace \varphi_{i}\rbrace_{i\in I}$ be a frame  for $\mathcal{H}$. A  Bessel sequence  $\psi = \lbrace \psi_{i}\rbrace_{i\in I}$ is called an  \textit{approximate dual frame} of $\varphi$, whenever $\Vert I_{\mathcal{H}} - T_{\psi}T_{\varphi}^{*}\Vert < 1$.
Applying Theorem 2.1 of \cite{Javanshiri} we obtain the next result.
\begin{prop}\label{p3.3}
Let $\varphi:=\{\varphi_{i}\}_{i\in I}$ be a redundant  frame for $\mathcal{H}$ and there exists an operator $T\in B(\mathcal{H})$ so that
\begin{eqnarray}\label{app2}
\Vert I_{\mathcal{H}}-T\Vert < 1,\quad  \Vert I_{\mathcal{H}}-T^{*}S_{\varphi}^{-1} \Vert^{2} < \dfrac{A_{\varphi}}{B_{\varphi}}.
\end{eqnarray}
 Then, $\varphi$ has  infinitely many  approximate dual frames such as  $\{\psi_{i}\}_{i\in I}$, for which $\{\varphi_{i}\}_{i\in I}$ and $\{\psi_{i}\}_{i\in I}$ are woven.
\end{prop}
\begin{proof}
Applying Theorem 2.1 of \cite{Javanshiri} the sequence $\{T^{*}S_{\varphi}^{-1}\varphi_{i}+\theta^{*}\delta_{i}\}_{i\in I}$ is an approximate dual of $\varphi$, in which $\{\delta_{i}\}_{i\in I}$ is the standard orthonormal basis of $l^{2}$ and the operator $\theta\in B(\mathcal{H}, l^{2})$ satisfies $T_{\varphi}\theta=0$. Also, by  Proposition \ref{prop1.4}, $\varphi$ and $\{T^{*}S_{\varphi}^{-1}\varphi_{i}\}_{i\in I}$ are woven with the universal lower bound $\left(\sqrt{A_{\varphi}}-\sqrt{B_{\varphi}}\Vert I_{\mathcal{H}}-T^{*}S_{\varphi}^{-1}\Vert\right)^{2}$. Using an argument similar to the proof of Theorem \ref{alter dual}  if we take $\epsilon>0$ so that
\begin{eqnarray*}
\epsilon^{2} \Vert \theta\Vert^{2}+2\epsilon\Vert \theta\Vert  \Vert S_{\varphi}^{-1}T\Vert \sqrt{B_{\phi}} < \left(\sqrt{A_{\varphi}}-\sqrt{B_{\varphi}}\Vert I_{\mathcal{H}}-T^{*}S_{\varphi}^{-1}\Vert\right)^{2},
\end{eqnarray*}
then  $\Psi_{\alpha}:=\{T^{*}S_{\varphi}^{-1}\varphi_{i}+\alpha \theta^{*}\delta_{i}\}_{i\in I}$ is an approximate dual frame of $\varphi$,  for all $0<\alpha<\epsilon$. Moreover, similar to the proof of Theorem \ref{alter dual} one can chek that  $\varphi$ and $\Psi_{\alpha}$ are woven frames.
\end{proof}

\begin{cor}Every  redundant Parseval frame  and also  $A$-tight frame  with $A>1/2$
 has  infinitely many   dual frames which are woven. Moreover, all $A$-tight frames have infinitely  many approximate dual frames which are woven with the original frame.
\end{cor}
\begin{ex}
 Suppose $\mathcal{H}=l^{2}$ with the standard orthonormal basis $\{\delta_{i}\}_{i=1}^{\infty}$ and consider the following known Parseval frame for  $\mathcal{H}$
\begin{eqnarray*}
\varphi=\left\{\delta_{1}, \dfrac{1}{\sqrt{2}}\delta_{2}, \dfrac{1}{\sqrt{2}}\delta_{2}, \dfrac{1}{\sqrt{3}}\delta_{3},   \dfrac{1}{\sqrt{3}}\delta_{3},  \dfrac{1}{\sqrt{3}}\delta_{3}, ... \right\}.
\end{eqnarray*}
Putting
\begin{eqnarray*}
U=\left\{0, \dfrac{1}{\sqrt{2}}\delta_{2}, -\dfrac{1}{\sqrt{2}}\delta_{2}, \dfrac{1}{2\sqrt{3}}\delta_{3},   \dfrac{1}{2\sqrt{3}}\delta_{3},  -\dfrac{1}{\sqrt{3}}\delta_{3},\dfrac{1}{3\sqrt{4}}\delta_{4},\dfrac{1}{3\sqrt{4}}\delta_{4},\dfrac{1}{3\sqrt{4}}\delta_{4},-\dfrac{1}{\sqrt{4}}\delta_{4},  ... \right\},
\end{eqnarray*}
one can easily check  that the Bessel  sequence $U$ satisfies (\ref{dual1}) with the upper bound $B_{U}=1$. Hence, $\varphi+\alpha U$ is a dual frame of  $\varphi$   for all $0< \alpha<\dfrac{1}{3}$,
and
\begin{eqnarray*}
\epsilon^{2} B_{U}+2\epsilon \sqrt{ B_{U}/A_{\varphi}} < 1,
\end{eqnarray*}
where $\epsilon =\dfrac{1}{3}$.
So, $\varphi$ and $\varphi+\alpha U$ are also  woven frames. Moreover, by taking  $\epsilon =\dfrac{1}{5}$, $T=\dfrac{1}{2}I_{\mathcal{H}}$, $\theta f=\{\langle f,u_{i}\rangle\}_{i\in I}$, for all $f\in \mathcal{H}$ and using Proposition \ref{p3.3} we see that  the Bessel sequence  $\dfrac{1}{2}\varphi+\alpha U$ constitutes an  approximate dual  of $\varphi$, which is also  woven with $\varphi$, for all $0< \alpha<\dfrac{1}{5}$.
\end{ex}

In the following theorem, we give  some sufficient  conditions  under which  a frame with its dual    are woven.
\begin{thm}\label{riesz cano}
Let $\varphi=\{\varphi_{i}\}_{i\in I}$ be a frame   for $\mathcal{H}$, then the following assertions hold:
\item [(i)]
if  $\psi=\{\psi_{i}\}_{i\in I}$  is a dual frame of $\varphi$  so that
for any $J \subset I$ the family $\{\varphi_{i}\}_{i\in J}\cup \{\psi_{i}\}_{i\in J^{c}}$ is a frame sequence.
Then $\varphi$  and $\psi$ are woven.\\
\item [(ii)] if $\varphi=\{\varphi_{i}\}_{i\in I}$ is  a Riesz basis   for $\mathcal{H}$, then $\varphi$ and its canonical dual are woven.
\end{thm}
\begin{proof}
Suppose $f\in \mathcal{H}$ so that $f\perp \{\varphi_{i}\}_{i\in J}\cup\{\psi_{i}\}_{i\in J^{c}}$. Then
\begin{eqnarray*}
\Vert f\Vert^{2} &=& \langle f, f\rangle = \left\langle f, \sum_{i\in I}\langle f, \psi_{i}\rangle \varphi_{i}\right\rangle \\
&=& \left\langle f, \sum_{i\in J}\langle f, \psi_{i}\rangle \varphi_{i}\right\rangle + \left\langle f, \sum_{i\in J^{c}}\langle f, \psi_{i}\rangle \varphi_{i}\right\rangle\\
&=& \sum_{i\in J}\langle f, \psi_{i}\rangle \langle f, \varphi_{i}\rangle + \sum_{i\in J^{c}}\langle f,\psi_{i}\rangle \langle f, \varphi_{i}\rangle =0.
\end{eqnarray*}
Hence, $f=0$ and consequently $\overline{\operatorname{span}} \left\{\{\varphi_{i}\}_{i\in J}\cup\{\psi_{i}\}_{i\in J^{c}}\right\}= \mathcal{H}$  for all $J\subseteq I$. This implies that $\varphi$ and $\psi$ are weakly  woven,  so Theorem \ref{2.0} follows the desired  result. To show $(ii)$, let $J\subset I$,  $X=\sum_{i\in J}a_{i}\varphi_{i}\in \overline{\operatorname{span}}\{\varphi_{i}\}_{i\in J}$ and $Y=\sum_{i\in J^{c}}b_{i}S_{\varphi}^{-1}\varphi_{i}\in \overline{\operatorname{span}}\{S_{\varphi}^{-1}\varphi_{i}\}_{i\in J^{c}}$ so that  $\Vert X\Vert = 1$ or $\Vert Y\Vert = 1$, then
 we obtain
\begin{eqnarray*}
&&\left\Vert X+Y\right\Vert^{2}\\
&=&\left\Vert \sum_{i\in J}a_{i}\varphi_{i}+\sum_{i\in J^{c}}b_{i}S_{\varphi}^{-1}\varphi_{i}\right\Vert^{2}\\
&=&\left\Vert \sum_{i\in J}a_{i}\varphi_{i}\right\Vert^{2}+\left\Vert \sum_{i\in J^{c}}b_{i}S_{\varphi}^{-1}\varphi_{i}\right\Vert^{2}+2Re \left\langle \sum_{i\in J}a_{i}\varphi_{i}, \sum_{i\in J^{c}}b_{i}S_{\varphi}^{-1}\varphi_{i}\right\rangle\\
&=&\left\Vert \sum_{i\in J}a_{i}\varphi_{i}\right\Vert^{2}+\left\Vert \sum_{i\in J^{c}}b_{i}S_{\varphi}^{-1}\varphi_{i}\right\Vert^{2} \geq 1.
\end{eqnarray*}
Thus, $\varphi$  and $\psi$  are woven by Theorem \ref{distance}. 
\end{proof}
The excess of a frame $\varphi$, denoted by $E(\varphi)$, is the greatest integer $m$ so that $m$ elements can be deleted from the frame  and still leave a frame, or $+\infty$ if there is no upper bound to the number of elements that can be removed. Moreover, it is known  that $E(\varphi)=\dim(Ker T_{\varphi})$ and zero excess frames are Riesz bases, see \cite{casaza balan2003}.
The above theorem, shows that every frame $\varphi$ which $E(\varphi)=0$ is woven with its canonical dual. Also, every frame $\varphi=\{\varphi_{i}\}_{i\in I}$  with $E(\varphi)=n$ can be written by $\varphi=\{\varphi_{i}\}_{i\in I\setminus \lbrace i_{1},... i_{n}\rbrace}\cup \{\varphi_{i_{1}}, ... \varphi_{i_{n
}}\}$, where $\{\varphi_{i}\}_{i\in I\setminus \lbrace i_{1},... i_{n}\rbrace}$ is a Riesz basis for $\mathcal{H}$ and  $\{\varphi_{i_{1}}, ... \varphi_{i_{n
}}\}$ are the  redundant elements of $\varphi$.
 In the next theorem we show that under some condition  frames are woven with their canonical duals.
\begin{thm}
Let $\varphi=\{\varphi_{i}\}_{i\in I}$ be a frame  for $\mathcal{H}$ so that the norm of its  redundant elements be small enough. Then $\varphi$  is woven with its canonical dual.
\end{thm}
\begin {proof}
First let $E(\varphi)=n$, for some $n\in \mathbb{N}$. 
Without loss of the  generality, we can write $\varphi=\{\varphi_{i}\}_{i\in I\setminus [n]}\cup \{\varphi_{i}\}_{i\in [n]}$, where $\phi=\{\varphi_{i}\}_{i\in I\setminus [n]}$ is a Riesz basis for $\mathcal{H}$. Hence,  by  Theorem \ref{riesz cano} $\phi$ and $S_{\phi}^{-1}\phi$ are woven and consequently  $\phi$ and $S_{\phi}\phi$ are woven  with a  universal lower bound  $A$, see Proposition 11 of \cite{lynch}. Now, 
 let 
\begin{eqnarray}\label{ttt}
\sum_{i\in [n]}\Vert \varphi_{i} \Vert^{2}< \sqrt{\dfrac{A}{B_{\varphi}}}.
\end{eqnarray}
We show that   $\varphi$ and its canonical dual    are also  woven. To this end, it is sufficient to prove the existence of a universal lower bound for frames $\varphi$ and $S_{\varphi}\varphi$.  Let $J\subseteq I$ and take $\alpha_{i,k}=\langle \varphi_{i}, \varphi_{k}\rangle$, $1\leq k\leq n$. Then
 \begin{eqnarray*}
&&\sum_{i\in J}\left \vert \alpha_{i,1}\langle f,\varphi_{1}\rangle+ ... + \alpha_{i,n}\langle f,\varphi_{n}\rangle \right\vert^{2}\\
&=&\sum_{i\in J}\left \vert (\alpha_{i,1}, ... , \alpha_{i,n}) . (\overline{\langle f,\varphi_{1}\rangle}, ... , \overline{\langle f,\varphi_{n}\rangle}) \right\vert^{2}\\
&\leq& \sum_{i=1}^{\infty} \left(\vert \alpha_{i,1}\vert^{2}+ ... + \vert \alpha_{i,n}\vert^{2})(\vert \langle f,\varphi_{1}\rangle\vert^{2}+ ... + \vert \langle f,\varphi_{n}\rangle\vert^{2}\right)\\
&\leq& \Vert f \Vert^{2}\left(\Vert \varphi_{1}\Vert^{2}+ ... + \Vert \varphi_{n}\Vert^{2}\right) \left( \sum_{i=1}^{\infty}\vert \langle \varphi_{i},\varphi_{1}\rangle\vert^{2}+ \vert \langle \varphi_{i},\varphi_{2}\rangle\vert^{2}+ ... + \vert \langle \varphi_{i},\varphi_{n}\rangle\vert^{2} \right)\\
&\leq& \Vert f \Vert^{2}\left(\Vert \varphi_{1}\Vert^{2}+ ... + \Vert \varphi_{n}\Vert^{2}\right)\left( B_{\varphi}\Vert \varphi_{1}\Vert^{2}+ ... + B_{\varphi}\Vert \varphi_{n}\Vert^{2} \right)\\
&=&\Vert f \Vert^{2}B_{\varphi}\left( \Vert \varphi_{1}\Vert^{2}+ ... + \Vert \varphi_{n}\Vert^{2}\right)^{2}.
\end{eqnarray*}
Now, without loss of the  generality, let $J\subseteq I\setminus [n]$, then we obtain
  \begin{eqnarray*}
&&\left(\sum_{i\in J}\vert \langle f, S_{\varphi}\varphi_{i}\rangle\vert^{2}+\sum_{i\in J^{c}}\vert \langle f, \varphi_{i}\rangle\vert^{2}\right)^{1/2}\\
&=&\left(\sum_{i\in J}\vert \langle f, S_{\varphi}\varphi_{i}\rangle\vert^{2}+\sum_{i\in J^{c}\setminus [n]}\vert \langle f, \varphi_{i}\rangle\vert^{2}+\sum_{i\in [n]}\vert \langle f, \varphi_{i}\rangle\vert^{2}\right)^{1/2}\\
&=&\left(\sum_{i\in J}\left\vert \langle f, S_{\phi}\varphi_{i}\rangle+\sum_{k\in [n]}\alpha_{i,k}\langle f, \varphi_{k}\rangle\right\vert^{2}+\sum_{i\in J^{c}\setminus [n]}\vert \langle f, \varphi_{k}\rangle\vert^{2}+\sum_{i\in [n]}\vert \langle f, \varphi_{i}\rangle\vert^{2}\right)^{1/2}\\
&\geq& \left( \sum_{i\in J}\vert \langle f, S_{\phi}\varphi_{i}\rangle\vert^{2}+\sum_{i\in J^{c}\setminus [n]}\vert \langle f, \varphi_{i}\rangle\vert^{2}+ \sum_{i\in [n]}\vert \langle f, \varphi_{i}\rangle\vert^{2}\right)^{1/2}\\
&-&\left( \sum_{i\in J}\left\vert \sum_{k\in [n]}\alpha_{i,k}\langle f, \varphi_{k}\rangle\right\vert^{2}\right)^{1/2}\\
&\geq& \left(A\Vert f\Vert^{2}+\sum_{i\in [n]}\vert \langle f, \varphi_{i}\rangle\vert^{2}\right)^{1/2}- \sqrt{B_{\varphi}}\left(\Vert \varphi_{1}\Vert^{2}+ ... + \Vert \varphi_{n}\Vert^{2}\right)\Vert f\Vert\\
&\geq& \sqrt{A}\Vert f\Vert-\sqrt{B_{\varphi}}\left(\Vert \varphi_{1}\Vert^{2}+ ... + \Vert \varphi_{n}\Vert^{2}\right)\Vert f\Vert\\
&=&\left(\sqrt{A}-\sqrt{B_{\varphi}}\sum_{i\in [n]}\Vert \varphi_{i} \Vert^{2}\right)\Vert f\Vert,
\end{eqnarray*}
for every $f\in \mathcal{H}$. Therefore,  the desired result follows provided that (\ref{ttt}) holds. On the other hand,  if $E(\varphi)=\infty$,  we can write $\varphi=\{\varphi_{i}\}_{i\in I\setminus \sigma}\cup \{\varphi_{i}\}_{i\in \sigma}$, where $\vert \sigma \vert=\infty$ and  $\phi=\{\varphi_{i}\}_{i\in I\setminus\sigma}$ is a Riesz basis for $\mathcal{H}$. Similar to the above computation one may  easily  check that   if 
\begin{eqnarray*}
\sum_{i\in \sigma}\Vert \varphi_{i} \Vert^{2}< \sqrt{\dfrac{A}{B_{\varphi}}}
\end{eqnarray*}
where $A$ is a  universal lower bound of woven frames  $\phi$ and $S_{\phi}\phi$. Then $\varphi$ is woven with its canonical dual
\end{proof}

\begin{thm}\label{last}
Let  $\varphi=\{\varphi_{i}\}_{i\in I}$ and $\psi=\{\psi_{i}\}_{i\in I}$  be two frames for $\mathcal{H}$.  The following hold:
\item [(i)]
If  $S_{\varphi}^{-1}\geq I_{\mathcal{H}}$ and $S_{\varphi}S_{\varphi_{J}}=S_{\varphi_{J}}S_{\varphi}$ for all $J\subseteq I$, then $\{\varphi_{i}\}_{i\in I}$ and $\{S_{\varphi}^{-1}\varphi_{i}\}_{i\in I}$ are woven.
\item[(ii)] If $\varphi$ and  $\psi$ are two woven  Riesz bases  and  $T_{1}$,  $T_{2}$ are invertible operators so that  
\begin{eqnarray*}
d_{\varphi_{J},\psi_{J^{c}}} > \max\left\{\Vert T_{1}-T_{2} \Vert \Vert T_{1}^{-1}\Vert, \Vert T_{1}-T_{2} \Vert \Vert T_{2}^{-1}\Vert \right\},  \quad (\textit{for all}  \quad J \subset I)
\end{eqnarray*}
where $d_{\varphi_{J},\psi_{J^{c}}}$  is defined as in Theorem \ref{distance}, then $T_{1}\varphi$ and $T_{2}\psi$ are woven.
\end{thm}
\begin{proof}
We first show $(i)$. Consider $\phi_{J}=\{\varphi_{i}\}_{i\in J}\cup \{S_{\varphi}^{-1}\varphi_{i}\}_{i\in J^{c}}$. Then $\phi_{J}$  is a Bessel sequence for  all $J\subseteq I$, and
\begin{eqnarray*}
S_{\phi_{J}}f&=& \sum_{i\in J} \langle f, \varphi_{i}\rangle \varphi_{i} + \sum_{i\in J^{c}} \langle f, S_{\varphi}^{-1}\varphi_{i}\rangle S_{\varphi}^{-1}\varphi_{i} \\
&=& S_{\varphi_{J}}f +S_{\varphi}^{-1}S_{\varphi_{J^{c}}}S_{\varphi}^{-1}f\\
&=& S_{\varphi}f-S_{\varphi_{J^{c}}}f +S_{\varphi}^{-1}S_{\varphi_{J^{c}}}S_{\varphi}^{-1}f\\
&=& S_{\varphi}f+ (S_{\varphi}^{-1}-I_{\mathcal{H}})S_{\varphi_{J^{c}}}(I_{\mathcal{H}}+S_{\varphi}^{-1})f
\end{eqnarray*}
for each $f\in \mathcal{H}$. Since $ (S_{\varphi}^{-1}-I_{\mathcal{H}})S_{\varphi_{J^{c}}}(I_{\mathcal{H}}+S_{\varphi}^{-1})$ is a positive operator, we have
\begin{eqnarray*}
S_{\phi_{J}}\geq S_{\varphi}.
\end{eqnarray*}
This implies that $T^{*}_{\phi_{J}}$ is injective and so $\phi_{J}$ is a frame   for all $J$. Hence, $(i)$ is obtained.
To show $(ii)$,  let $J\subset I$, $f=\sum_{i\in J}c_{i}T_{1}\varphi_{i}$ and $g=\sum_{i\in J^{c}}d_{i}T_{2}\psi_{i}$. Suppose $\Vert g\Vert = 1$ then
\begin{eqnarray*}
\Vert f-g\Vert &=& \left\Vert \sum_{i\in J}c_{i}T_{1}\varphi_{i} -\sum_{i\in J^{c}}d_{i}T_{2}\psi_{i} \right\Vert\\
&=& \left\Vert \sum_{i\in J}c_{i}T_{1}\varphi_{i} - \sum_{i\in J^{c}}d_{i}T_{1}\psi_{i}+ \sum_{i\in J^{c}}d_{i}T_{1}\psi_{i} -\sum_{i\in J^{c}}d_{i}T_{2}\psi_{i} \right\Vert\\
&\geq& \left\Vert T_{1}\left( \sum_{i\in J}c_{i}\varphi_{i}-\sum_{i\in J^{c}}d_{i}\psi_{i} \right)\right\Vert - \left\Vert (T_{1}-T_{2})\sum_{i\in J^{c}}d_{i}\psi_{i}\right\Vert \\
&\geq&  \left\Vert T_{1}^{-1}\right\Vert^{-1} \left\Vert \sum_{i\in J^{c}}d_{i}\psi_{i}\right\Vert  \left\Vert \dfrac{\sum_{i\in J}c_{i}\varphi_{i}}{\left\Vert \sum_{i\in J^{c}}d_{i}\psi_{i}\right\Vert}-\dfrac{\sum_{i\in J^{c}}d_{i}\psi_{i}}{\left\Vert \sum_{i\in J^{c}}d_{i}\psi_{i}\right\Vert }\right\Vert\\
&-& \Vert T_{1}-T_{2}\Vert \left\Vert \sum_{i\in J^{c}}d_{i}\psi_{i}\right\Vert \\
&\geq& \left( d_{\varphi_{J},\psi_{J^{c}}}\Vert T_{1}^{-1}\Vert^{-1}-\Vert T_{1}-T_{2}\Vert\right)  \left\Vert \sum_{i\in J^{c}}d_{i}\psi_{i}\right\Vert\\
&\geq& \left( d_{\varphi_{J},\psi_{J^{c}}}\Vert T_{1}^{-1}\Vert^{-1}-\Vert T_{1}-T_{2}\Vert\right)\Vert T_{2}\Vert^{-1}>0.
\end{eqnarray*}
On the other hand, if $\Vert f\Vert = 1$ then similarly we can write
\begin{eqnarray*}
\Vert f-g\Vert
&=& \left\Vert \sum_{i\in J}c_{i}T_{1}\varphi_{i} - \sum_{i\in J}c_{i}T_{2}\varphi_{i}+ \sum_{i\in J}c_{i}T_{2}\varphi_{i} -\sum_{i\in J^{c}}d_{i}T_{2}\psi_{i} \right\Vert\\
&\geq& \left\Vert T_{2}\left( \sum_{i\in J}c_{i}\varphi_{i}-\sum_{i\in J^{c}}d_{i}\psi_{i} \right)\right\Vert - \left\Vert (T_{1}-T_{2})\sum_{i\in J}c_{i}\varphi_{i}\right\Vert \\
&\geq& \left( d_{\varphi_{J},\psi_{J^{c}}}\Vert T_{2}^{-1}\Vert^{-1}-\Vert T_{1}-T_{2}\Vert\right)\Vert T_{1}\Vert^{-1}>0.
\end{eqnarray*}
Hence, by considering
\begin{eqnarray*}
d_{1}=\left( d_{\varphi_{J},\psi_{J^{c}}}\Vert T_{1}^{-1}\Vert^{-1}-\Vert T_{1}-T_{2}\Vert\right)\Vert T_{2}\Vert^{-1},
\end{eqnarray*}
\begin{eqnarray*}
d_{2}=\left( d_{\varphi_{J},\psi_{J^{c}}}\Vert T_{2}^{-1}\Vert^{-1}-\Vert T_{1}-T_{2}\Vert\right)\Vert T_{1}\Vert^{-1},
\end{eqnarray*}
and taking
$c:=\min\{ d_{1}, d_{2}\}$
we obtain $d_{T_{1}\varphi_{J},T_{2}\psi_{J^{c}}}\geq c>0$. Thus, the result follows by Theorem  \ref{distance}.
\end{proof}
We know that applying two different operators to woven frames can give frames which are not woven, see Example 2 of \cite{Bemros} . The above theorem  presents   a condition  where different operators   preserve the weaving property, and the canonical duals of two woven frames are woven  as a consequence.
\begin{cor}
Let $\{\varphi_{i}\}_{i\in I}$ and $\psi=\{\psi_{i}\}_{i\in I}$  be two woven  Riesz bases   for $\mathcal{H}$ so that for every $J\subset I$
\begin{eqnarray*}
d_{\varphi_{J},\psi_{J^{c}}} > \max\left\{\Vert S_{\varphi}^{-1}-S_{\psi}^{-1} \Vert \Vert S_{\varphi}\Vert, \Vert S_{\varphi}^{-1}-S_{\psi}^{-1} \Vert \Vert S_{\psi}\Vert \right\}.
\end{eqnarray*}
Then $S_{\varphi}^{-1}\varphi$ and $S_{\psi}^{-1}\psi$ are also woven.
\end{cor}
\section{ Woven frames and perturbations}
In this section, we give some conditions which frames with their perturbations constitute woven frames.
  Proposition \ref{prop1.4} states perturbed frames as the image of an invertible operator  $T$ of    a given frame.
In the following,  we obtain some other  invertible operators  $T$ for which   $\varphi$ and $T\varphi$ are woven frames. First, we recall a notion of \cite{cahil}.
Let  $\Lambda = \{\lambda_{i}\}_{i=1}^{n}$ and $\Omega=\{\gamma_{i}\}_{i=1}^{n}$ be orthonormal bases for $\mathcal{H}_{n}$. Also, let  $\alpha=\{\alpha_{i}\}_{i=1}^{n}$  and $\beta=\{\beta_{i}\}_{i=1}^{n}$ be  sequences of positive constants. An operator $T:\mathcal{H}_{n}\rightarrow \mathcal{H}_{n}$ is called admissible for $( \Lambda, \Omega, \alpha, \beta)$, if there exists an orthonormal basis $\{e_{i}\}_{i=1}^{n}$ for $\mathcal{H}_{n}$ satisfying
\begin{eqnarray*}
T^{*}\lambda_{i}=\sum_{k=1}^{n} \sqrt{\dfrac{\alpha_{i}}{\beta_{k}}}\langle e_{i}, \gamma_{k}\rangle \gamma_{k}, \quad (i=1,2,..., n).
\end{eqnarray*}
Moreover,  in the above definition if $\Lambda=\Omega$ and $\alpha=\beta$ we say that $T$ is admissible for $( \Lambda,  \alpha)$.

Now,
suppose
$\varphi=\{\varphi_{i}\}_{i=1}^{m}$ is a frame   for $n$-dimensional Hilbert space $\mathcal{H}_{n}$. Also, let $\Lambda=\lbrace \lambda_{i}\rbrace_{i=1}^{n}$ be eigenvectors of $S_{\varphi}$ with eigenvalues $\alpha=\lbrace \alpha_{i}\rbrace_{i=1}^{n}$. If $T$ is an invertible operator on $\mathcal{H}_{n}$, which is admissible for $(\Lambda,\alpha)$ then $\varphi$ and $T\varphi=\{T\varphi_{i}\}_{i=1}^{n}$ are woven.
More precisely, since  $T$ is admissible for $(\Lambda,\alpha)$ so $S_{\varphi}=TS_{\varphi}T^{*}$, by  Theorem 2.12 of \cite{cahil}. On the other hand, $T\varphi$ is a frame for $\mathcal{H}_{n}$ with the frame operator $S_{T\varphi}=TS_{\varphi}T^{*}$. Hence,
for all $J\subseteq \{1,2,...,M\}$ and by considering $\phi_{J}=\{\varphi_{i}\}_{i\in J}\cup \{T\varphi_{i}\}_{i\in J^{c}}$
we obtain
\begin{eqnarray*}
S_{\phi_{J}}=S_{\varphi_{J}}+S_{T\varphi_{J^{c}}}=S_{\varphi},
\end{eqnarray*}
where $S_{\varphi_{J}}$ and $S_{T\varphi_{J^{c}}}$ are  the frame operators of $\{\varphi_{i}\}_{i\in J}$ and $\{T\varphi_{i}\}_{i\in J^{c}}$, respectively. Thus $S_{\phi_{J}}$, the frame operator of $\phi_{J}$,  is a positive and  invertible operator on $\mathcal{H}_{n}$ and so $\phi_{J}$ is a frame, i.e., $\varphi$ and $T\varphi$ are woven.

 The first results on perturbation of woven frames was given in \cite{Bemros}, then authors in \cite{vash.dep, Vashisht} extended some of these results.
In the sequel, we also  extend     Theorem 6.1 of  \cite{Bemros} to obtaion some new results.
\begin{thm}\label{perturbation1}
Suppose $\varphi=\{\varphi_{i}\}_{i\in I}$ and $\psi=\{\psi_{i}\}_{i\in I}$ are two
frames so that for all sequences of scalars $\{c_{i}\}_{i\in I}$ we have
\begin{eqnarray*}
\left\Vert \sum_{i\in I}c_{i}(\varphi_{i}-\psi_{i})\right\Vert \leq \lambda_{1}\left\Vert \sum_{i\in I}c_{i}\varphi_{i}\right\Vert+\lambda_{2}\left\Vert \sum_{i\in I}c_{i}\psi_{i}\right\Vert+\mu \left( \sum_{i\in I}\vert c_{i}\vert^{2} \right)^{1/2},
\end{eqnarray*}
for some positive numbers $\lambda_{1}, \lambda_{2}, \mu$, where
\begin{eqnarray*}
\lambda_{1}\sqrt{B_{\varphi}}+\lambda_{2}\sqrt{B_{\psi}}+ \mu <\sqrt{A_{\varphi}}.
\end{eqnarray*}
Then, $\varphi$ and $\psi$ are woven.
\end{thm}
\begin{proof}
Consider $T_{\varphi_{J^{c}}}$ and $T_{\psi_{J^{c}}}$ as the synthesis operators of   Bessel sequences $\{\varphi_{i}\}_{i\in J^{c}}$ and $\{\psi_{i}\}_{i\in J^{c}}$, respectively. Then for every $J\subseteq I$ and $f\in \mathcal{H}$, we have
\begin{eqnarray*}
&&\left(\sum_{i\in J}|\langle f,\varphi_{i}\rangle|^{2}+\sum_{i\in J^{c}}|\langle f,\psi_{i}\rangle|^{2} \right)^{1/2}\\
&=& \left(\sum_{i\in J}|\langle f,\varphi_{i}\rangle|^{2}+\sum_{i\in J^{c}}|\langle f,\varphi_{i}\rangle-\langle f,\varphi_{i}-\psi_{i}\rangle|^{2}\right)^{1/2}\\
&\geq&   \left(\sum_{i\in I}|\langle f,\varphi_{i}\rangle|^{2}\right)^{1/2}-\left(\sum_{i\in J^{c}}|\langle f,\varphi_{i}-\psi_{i}\rangle|^{2}\right)^{1/2}\\
&\geq& \sqrt{A_{\varphi}}\left\Vert f\right\Vert-\Vert T_{\varphi_{J^{c}}} f-T_{\psi_{J^{c}}}f\Vert \\
&\geq& \left(\sqrt{A_{\varphi}}-\Vert T_{\varphi}-T_{\psi}\Vert\right) \Vert f\Vert\\
&\geq& \left(\sqrt{A_{\varphi}}-\lambda_{1}\Vert T_{\varphi}\Vert -\lambda_{2}\Vert T_{\psi}\Vert-\mu \right) \Vert f\Vert\\
&\geq& \left(\sqrt{A_{\varphi}}-\lambda_{1}\sqrt{B_{\varphi}} -\lambda_{2}\sqrt{B_{\psi}}-\mu \right) \Vert f\Vert.
\end{eqnarray*}
Using the assumption $\sqrt{A_{\varphi}}-\lambda_{1}\sqrt{B_{\varphi}} -\lambda_{2}\sqrt{B_{\psi}}-\mu> 0$, the lower bound is obtained.
Also, clearly $\{\varphi_{i}\}_{i\in J}\cup \{\psi_{i}\}_{i\in J^{c}}$ is Bessel with an upper bound $B_{\varphi}+B_{\psi}$. Thus the result follows.
\end{proof}
\begin{rem}
It is worth to note that, Theorem 6.1 of  \cite{Bemros} is obtained from the above theorem in case $\lambda_{1}=\lambda_{2}=0$, with the  perturbation bound
\begin{eqnarray*}
\mu \leq \dfrac{A_{\varphi}}{2(\sqrt{B_{\varphi}}+\sqrt{B_{\psi}})} <  \dfrac{\sqrt{A_{\varphi}}}{2}.
\end{eqnarray*}
Hence, in  Theorem \ref{perturbation1}  we have more freedom for perturbation bound.
\end{rem}
Applying Theorem \ref{perturbation1}, we will obtain the following results.
\begin{cor}
Suppose $\varphi=\{\varphi_{i}\}_{i\in I}$ is a frame  for $\mathcal{H}$ and $0 \neq h\in \mathcal{H}$. Also, let $\lbrace \lambda_{i}\rbrace_{i\in I}$ be a sequence of scalars so that
\begin{eqnarray*}
\sum_{i\in I}\vert \lambda_{i}\vert^{2} <\alpha\dfrac{A_{\varphi}}{\Vert h\Vert^{2}},
\end{eqnarray*}
for some $\alpha < 1$.
 Then $\{\varphi_{i}\}_{i\in I}$ and $\{\varphi_{i}+\lambda_{i}h\}_{i\in I}$ are woven.
\end{cor}
\begin{proof}
Obviously,  $\{\varphi_{i}+\lambda_{i}h\}_{i\in I}$ is a Bessel sequence with the upper bound $\left(\sqrt{ B_{\varphi}}+\Vert \{\lambda_{i}\}_{i\in I}\Vert\Vert h\Vert\right)^{2}$. Moreover, for any sequence $\lbrace c_{i}\rbrace_{i\in I}$ of scalars
\begin{eqnarray*}
\left\Vert \sum_{i\in I}c_{i}(\varphi_{i}+\lambda_{i}h-\varphi_{i})\right\Vert &=&\left\Vert \sum_{i\in I}c_{i}\lambda_{i}h\right\Vert \\
&\leq& \sum_{i\in I}\vert c_{i}\vert \vert \lambda_{i}\vert \Vert h\Vert\\
&\leq&\left(\sum_{i\in I}\vert c_{i}\vert^{2}\right)^{1/2}\left(\sum_{i\in I}\vert \lambda_{i}\vert^{2}\right)^{1/2}\Vert h\Vert\\
&<&\sqrt{\alpha A_{\varphi}}\left(\sum_{i\in I}\vert c_{i}\vert^{2}\right)^{1/2}.
\end{eqnarray*}
So, the result follows by Theorem \ref{perturbation1}.
\end{proof}
\begin{cor}
Suppose $m, n\in \mathbb{N}$ are relatively prime and $\varphi=\{\varphi_{i}\}_{i=1}^{n}$ is an  $\epsilon$-nearly equal norm and $\epsilon$-nearly Parseval frame for $\mathcal{H}_{m}$  for $\epsilon>0$ small enough. Then there exists an equal-norm Parseval frame $\psi=\{\psi_{i}\}_{i=1}^{n}$ so that $\varphi$ and $\psi$ are woven.
\end{cor}
\begin{proof}
Considering
\begin{eqnarray*}
0 < \epsilon < \min \left\{\dfrac{1}{2}, \dfrac{8\alpha\sqrt{A_{\varphi}}}{4\sqrt{m}+27m^{2}n(n-1)^{8}}\right\},
\end{eqnarray*}
for some $\alpha<1$. There exists an equal-norm Parseval frame $\psi=\{\psi_{i}\}_{i=1}^{n}$ so that
\begin{eqnarray*}
\left(\sum_{i=1}^{n}\Vert \varphi_{i}-\psi_{i}\Vert^{2}\right)^{1/2} \leq \dfrac{\sqrt{m}}{2}\epsilon+\dfrac{27}{8}m^{2}n(n-1)^{8}\epsilon.
\end{eqnarray*}
See \cite{BC}.
Thus, for all $\{c_{i}\}_{i=1}^{n}$ of scalars we have
\begin{eqnarray*}
\left\Vert \sum_{i=1}^{n}c_{i}(\varphi_{i}-\psi_{i})\right\Vert &\leq& \left(\sum_{i=1}^{n}\vert c_{i}\vert^{2}\right)^{1/2}
\left(\sum_{i=1}^{n}\Vert \varphi_{i}-\psi_{i}\Vert^{2}\right)^{1/2}\\
&\leq& \left(\dfrac{\sqrt{m}}{2}\epsilon+\dfrac{27}{8}m^{2}n(n-1)^{8}\epsilon\right) \left(\sum_{i=1}^{n}\vert c_{i}\vert^{2}\right)^{1/2}\\
&<&\alpha \sqrt{ A_{\varphi}}\left(\sum_{i=1}^{n}\vert c_{i}\vert^{2}\right)^{1/2}.
\end{eqnarray*}
Using Theorem \ref{perturbation1} we obtain the desired result.
\end{proof}

\bibliographystyle{amsplain}

\end{document}